\documentclass[11pt,reqno]{amsart}

\usepackage[T1]{fontenc}
\usepackage{amsmath,amssymb,mathtools}
\usepackage{microtype}
\usepackage[hidelinks]{hyperref}

\allowdisplaybreaks
\theoremstyle{plain}
\newtheorem{theorem}{Theorem}[section]
\newtheorem{proposition}[theorem]{Proposition}
\newtheorem{lemma}[theorem]{Lemma}
\newtheorem{corollary}[theorem]{Corollary}

\theoremstyle{definition}
\newtheorem{problem}[theorem]{Problem}

\theoremstyle{remark}
\newtheorem{remark}[theorem]{Remark}

\newcommand{\F}{\mathcal F}
\newcommand{\D}{\mathcal D}
\newcommand{\Hh}{\mathcal H}
\newcommand{\Vv}{\mathcal V}
\newcommand{\R}{\mathbb R}
\newcommand{\HH}{\mathbb H}
\newcommand{\CP}{\mathbb{CP}}
\newcommand{\Id}{\operatorname{Id}}
\newcommand{\spanop}{\operatorname{span}}

\newcommand{\Kmix}{K_{\mathrm{mix}}}
\newcommand{\rh}{\rho}

\title[Closed counterexamples to Toponogov's question]
{Closed counterexamples to Toponogov's question on mixed curvature}

\author{Yaroslav Bazaikin}
\address{Department of Mathematics, Faculty of Science, Jan Evangelista
Purkyn\v{e} University in \'{U}st\'{\i} nad Labem,
Pasteurova 3632/15, 400 96 \'{U}st\'{\i} nad Labem,
Czech Republic}
\email{bazaikin@gmail.com}

\subjclass[2020]{53C12, 53C20, 53C21, 57R30}
\keywords{totally geodesic foliations, mixed sectional curvature,
Toponogov's problem, conullity tensor, fatness, Riccati equation}

\hypersetup{
  pdftitle={Closed counterexamples to Toponogov's question on mixed curvature},
  pdfauthor={Yaroslav Bazaikin},
  pdfsubject={Totally geodesic foliations and positive mixed curvature},
  pdfkeywords={mixed curvature, totally geodesic foliations, Toponogov's problem, Riccati equation}
}

\begin{document}

\begin{abstract}
We give a negative answer to Toponogov's question whether positive mixed
sectional curvature on a closed manifold forces the Ferus--Adams dimension
bound for a totally geodesic foliation.  For every even \(N\ge4\), we
construct explicit metrics on \(S^{4m-1}\times S^3\) and on totally geodesic
fixed point submanifolds \(S^{4m-1}\times S^1\).  The leaves are closed
geodesics forming smooth circle fibrations, but the metrics are not
bundle-like.  One subfamily has a positive mixed-curvature operator that is
parallel along every leaf but nonscalar, with spectrum
\(\{\kappa,4\kappa\}\); it is a
closed realization of Rovenskii's local anisotropic Riccati model.  In a
second subfamily, \(\delta_{\mathrm{mix}}\to1\) while the curvature operator
remains nonparallel.  After normalizing \(\max\Kmix=1\), the leaf lengths tend
to \(2\pi\), thereby answering the bounded-length question for the local
pinched models in the closed setting.
Thus neither leafwise parallelism nor pinching by any fixed
constant below one replaces scalarity in the Ferus argument;
by contrast, the Ferus--Adams bound remains valid in the
bundle-like case.
\end{abstract}

\maketitle

\section{Introduction}
\label{sec:introduction}

Throughout the paper, closed means compact without boundary.  Let
\((M^{p+n},g)\) carry a regular totally geodesic foliation \(\F^p\).  A mixed
plane is spanned by a leafwise and a normal vector, and its sectional
curvature is denoted by \(\Kmix\).  Ferus proved that if all mixed sectional
curvatures along every complete leaf equal the same positive constant, then
\begin{equation}\label{eq:Ferus-bound}
        p\le \rh(n)-1,
\end{equation}
where \(\rh(n)-1\) is the maximal number of pointwise linearly independent
vector fields on \(S^{n-1}\) \cite{Adams1962,Ferus1970}.  Toponogov asked
whether \eqref{eq:Ferus-bound} still follows on a closed manifold from the
weaker assumption \(\Kmix>0\).  For background on mixed curvature and the
extrinsic geometry of foliations, see
\cite{Rovenskii2000,RovenskiWalczak2021,Rovenski2023} and the references
therein.  Local examples and refinements involving closed geodesics,
\(L\)-Jacobi fields, and partial Ricci flow appear in
\cite{Rovenskii1982,Rovenskii1994,Rovenskii1997,
RovenskiSharafutdinov2013,Rovenski2014,Rovenski2019}.

There is an important positive case.  If the metric is bundle-like, the
foliation is locally a Riemannian submersion with totally geodesic fibers.
O'Neill's formula identifies positive mixed curvature with fatness, and the
usual fatness argument produces \(p\) independent vector fields on
\(S^{n-1}\); hence \eqref{eq:Ferus-bound} holds.  We include the short
argument in Proposition~\ref{prop:bundle-like-Ferus}.  Thus any
counterexample must leave the bundle-like category.

The role of the constant-curvature hypothesis is seen from the Riccati
equation along a leafwise geodesic,
\[
        B'+B^2+R_T=0,
\]
where \(B\) is the conullity tensor and \(R_T\) the mixed-curvature operator.
When \(R_T=k\Id\), scalarity preserves real eigenspaces of \(B\), and the
matrix equation restricts to a scalar Riccati equation, which cannot exist
along a complete geodesic.  Positivity alone does not provide this invariant
line.

Both local ways in which the scalar argument can fail were already present
in Rovenskii's work.  In odd normal rank, the constant positive matrix
\[
        \operatorname{diag}\left(1,\frac14,\ldots,\frac14\right)
\]
admits a global matrix Riccati solution and gives a local foliation near a
closed geodesic \cite[Eq.~(3.12)]{Rovenskii2000}.  The same survey constructs
almost \(1\)-pinched tubular models on \(S^1\times B^n(r)\), whose normalized
leaf lengths diverge, and asks whether those lengths can remain bounded
\cite[Sec.~3.4.2]{Rovenskii2000}.  These local models did not provide closed
total spaces; the compact problem was still recorded as open in later work
\cite{RovenskiSharafutdinov2013}.  The point of the construction below is to
realize both mechanisms globally while keeping positive mixed curvature in
all normal directions.

\begin{theorem}[Closed counterexamples]\label{thm:main}
For every even integer \(N\ge4\), there is a closed \(N\)-dimensional
Riemannian manifold carrying a one-dimensional totally geodesic foliation
with \(\Kmix>0\).  Its leaves are closed geodesics and form a smooth circle
fibration, but the metric is not bundle-like with respect to this fibration.
\end{theorem}

The examples are metrics on \(S^{4m-1}\times S^3\), in dimensions \(4m+2\),
and their restrictions to totally geodesic fixed point submanifolds
\(S^{4m-1}\times S^1\), in dimensions \(4m\).  The leaf dimension is
\(p=1\), while \(n=N-1\) is odd and hence \(\rh(n)-1=0\).  They therefore
contradict \eqref{eq:Ferus-bound} in every even dimension.

The first refinement is a closed realization of the local parallel
anisotropic model.  The derivative below is taken in
\(\operatorname{End}(T^\perp)\) with the normal connection.

\begin{theorem}[Parallel anisotropic counterexamples]
\label{thm:parallel-anisotropic}
For every even \(N\ge4\), the examples in Theorem~\ref{thm:main} can be
chosen so that
\[
        \nabla_T^\perp R_T=0,
        \qquad
        \operatorname{spec}(R_T)=\{\kappa,4\kappa\}
\]
for a constant \(\kappa>0\).  The eigenvalue \(4\kappa\) is simple and
\(\kappa\) has multiplicity \(N-2\).
\end{theorem}

Thus leafwise parallelism of a positive mixed-curvature operator does not
replace scalarity.  A complementary subfamily approaches the scalar regime
pointwise.  Set
\[
        \delta_{\mathrm{mix}}(g,\F)
        =\frac{\min\Kmix}{\max\Kmix},
\]
with the extrema over mixed unit planes.

\begin{theorem}[Almost \(1\)-pinched counterexamples]
\label{thm:pinched}
For every even \(N\ge4\) and every \(0<\delta<1\), the examples in
Theorem~\ref{thm:main} can be chosen with
\(\delta_{\mathrm{mix}}(g,\F)>\delta\).  Along the explicit almost
\(1\)-pinched families, after normalizing \(\max\Kmix=1\), the circle lengths
converge to \(2\pi\).
\end{theorem}

Consequently, no fixed pinching constant \(\delta<1\) restores the
Ferus--Adams estimate outside the bundle-like setting.  In the first
subfamily \(R_T\) is parallel but anisotropic; in the second it is almost
scalar but its eigenspaces rotate in a parallel normal frame.  The final
section uses the symmetric part of the conullity tensor to formulate a
quantitative problem separating these examples from the bundle-like case.

\section{The Riccati mechanism and the bundle-like case}
\label{sec:preliminaries}

Write \(\Vv=T\F\) and \(\Hh=\Vv^\perp\).  We use
\[
R(X,Y)Z=\nabla_X\nabla_YZ-\nabla_Y\nabla_XZ-\nabla_{[X,Y]}Z,
\]
so the round sphere has positive curvature.  For a totally geodesic
foliation and \(X\in\Vv\), the conullity tensor is
\[
        B_X(U)=(\nabla_UX)^{\Hh},\qquad U\in\Hh.
\]
Some authors use the opposite sign.  In the one-dimensional case, let \(T\)
be the unit leaf field.  Then \(\nabla_TT=0\), \(B(U)=\nabla_UT\), and
\[
        R_TU=R(U,T)T,
        \qquad
        \Kmix(T,U)=\frac{\langle R_TU,U\rangle}{|U|^2}.
\]
Thus \(\Kmix>0\) is equivalent to \(R_T>0\).

\begin{lemma}[Riccati equation]\label{lem:riccati}
Along every integral curve of \(T\), after normal parallel identification,
\[
        B'+B^2+R_T=0.
\]
\end{lemma}

\begin{proof}
Vary the integral curve through a transverse curve with initial variation
\(U_0\perp T\).  The resulting variational field \(J\) is a normal Jacobi
field, \([T,J]=0\), and \(\nabla_TJ=\nabla_JT=BJ\).  The Jacobi equation gives
\((B'+B^2+R_T)J=0\); varying \(U_0\) proves the identity.
\end{proof}

Suppose now that \(R_T=k\Id\), \(k>0\), along a complete leaf.  If \(B(0)\)
has a real eigenvalue \(\lambda_0\), then
\(C(t)=[B(t),B(0)]\) satisfies
\[
        C'=-BC-CB,\qquad C(0)=0.
\]
Hence \(B(t)\) preserves the corresponding real eigenspace.  Its restriction
starts at \(\lambda_0\Id\) and, by uniqueness, satisfies the scalar equation
\(\lambda'+\lambda^2+k=0\), whose solutions have a finite-time pole.  Thus
\(B(0)\) has no real eigenvalue.  In odd normal rank this is impossible; in
higher leaf dimension the same mechanism together with Adams' theorem gives
Ferus' bound.

\subsection{Moving and parallel normal frames}

Let \(N_1,\ldots,N_n\) be an orthonormal normal frame along a leaf and write
\(\nabla_TN_j=\sum_i\Omega_{ij}N_i\), with
\(\Omega^{\mathsf T}=-\Omega\).  If \(B_0\) and \(R_0\) are the matrices in
this frame, then
\[
        B_0'+[\Omega,B_0]+B_0^2+R_0=0.
\]
For the solution \(P'=-\Omega P\), \(P(0)=\Id\), the parallel-frame matrices
are
\[
        B=P^{\mathsf T}B_0P,
        \qquad
        R_T=P^{\mathsf T}R_0P,
\]
and
\begin{equation}\label{eq:parallel-derivative}
        \frac{d}{dt}R_T
        =P^{\mathsf T}(R_0'+[\Omega,R_0])P.
\end{equation}
Equivalently, for a normal field \(U\),
\[
 ((\nabla_T^\perp R_T)U)
 =\nabla_T^\perp(R_TU)-R_T(\nabla_T^\perp U).
\]
In the homogeneous frames below, \(B_0,R_0,\Omega\) are constant.  Hence
\(R_T\) is parallel precisely when \([\Omega,R_0]=0\); otherwise
\(R_T(t)=e^{t\Omega}R_0e^{-t\Omega}\).

\subsection{Bundle-like metrics}

For a one-dimensional geodesic foliation,
\begin{equation}\label{eq:Lie-B}
        (\mathcal L_Tg)(U,V)
        =\langle BU,V\rangle+\langle U,BV\rangle
        \qquad(U,V\perp T).
\end{equation}
Thus the metric is bundle-like if and only if \(B\) is skew-symmetric.
The following standard argument records the positive side of Toponogov's
question \cite[Sec.~3.6.3]{Rovenskii2000}.

\begin{proposition}[Bundle-like case]\label{prop:bundle-like-Ferus}
Let \(\F^p\) be a bundle-like totally geodesic foliation on
\(M^{p+n}\).  If \(\Kmix>0\), then \(p\le\rh(n)-1\).
\end{proposition}

\begin{proof}
Locally the foliation is a Riemannian submersion with totally geodesic
fibers.  Fix a point, an orthonormal vertical basis
\(V_1,\ldots,V_p\), and let \(\mathcal A\) be O'Neill's tensor.
For \(X\in S(\Hh)\cong S^{n-1}\), set
\[
        Z_i(X)=\mathcal A_XV_i.
\]
Since \(\mathcal A_X\) is skew-adjoint and \(\mathcal A_XX=0\),
\(\langle Z_i(X),X\rangle
=-\langle V_i,\mathcal A_XX\rangle=0\);
hence \(Z_i(X)\in T_XS(\Hh)\).
O'Neill's mixed-curvature formula gives
\(\Kmix(X,V)=|\mathcal A_XV|^2\) for unit \(X,V\).  Hence a relation
\(\sum_i a_iZ_i(X)=0\) would imply
\(\mathcal A_X(\sum_i a_iV_i)=0\), and positivity forces all \(a_i=0\).
Thus \(Z_1,\ldots,Z_p\) are pointwise independent vector fields on
\(S^{n-1}\), and Adams' theorem yields the claim
\cite{ONeill1966,Weinstein1980,FloritZiller2011}.
\end{proof}

\section{Hopf--Berger data on quaternionic spheres}
\label{sec:hopf-berger-data}

Let \(g_0\) be the round metric of curvature \(1\) on
\(S^{4m-1}\subset\HH^m\), and let
\[
        \xi_1(x)=xi,\qquad \xi_2(x)=xj,\qquad \xi_3(x)=xk.
\]
These orthonormal Killing fields satisfy
\[
 [\xi_1,\xi_2]=2\xi_3,\qquad
 [\xi_2,\xi_3]=2\xi_1,\qquad
 [\xi_3,\xi_1]=2\xi_2.
\]
Put \(\eta_\alpha=g_0(\xi_\alpha,\cdot)\) and
\[
        \D=\bigcap_{\alpha=1}^3\ker\eta_\alpha.
\]
Right multiplication induces orthogonal skew endomorphisms
\(J_1,J_2,J_3\) of \(\D\).  For \(X,Y\in\Gamma(\D)\),
\begin{equation}\label{eq:hopf-horizontal-brackets}
 \langle[X,Y],\xi_\alpha\rangle_{g_0}
 =-2\langle J_\alpha X,Y\rangle_{g_0},
 \qquad
 \langle[\xi_\alpha,X],\xi_\beta\rangle_{g_0}=0.
\end{equation}
For \(m=1\), \(\D=0\).

Fix \(A>C>0\), set
\[
        L=\frac2{\sqrt{AC}},\qquad M=\frac2A,
\]
and define
\[
        g_{A,C}=L^2g_0+(M^2-L^2)\eta_2^2.
\]
The fields
\[
 u_1=\frac{\sqrt{AC}}2\xi_1,\qquad
 u_2=\frac A2\xi_2,\qquad
 u_3=\frac{\sqrt{AC}}2\xi_3
\]
are \(g_{A,C}\)-orthonormal and satisfy
\begin{equation}\label{eq:berger-brackets}
 [u_1,u_2]=Au_3,\qquad [u_2,u_3]=Au_1,\qquad [u_3,u_1]=Cu_2.
\end{equation}
Write \(\beta_\alpha=g_{A,C}(u_\alpha,\cdot)\).  Koszul's formula and
\eqref{eq:berger-brackets} give \(\nabla_{u_1}u_1=0\).  Since the
\(\xi_1\)-flow has period \(2\pi\), the unit \(u_1\)-orbits have length
\begin{equation}\label{eq:leaf-length}
        \ell_{\mathrm{leaf}}=\frac{4\pi}{\sqrt{AC}}.
\end{equation}

Along a \(u_1\)-orbit choose a \(u_1\)-invariant orthonormal frame of
\(\D\).  After the above scaling, \eqref{eq:hopf-horizontal-brackets} becomes
\begin{equation}\label{eq:D-bracket-u1}
 \langle[X,Y],u_1\rangle_{g_{A,C}}
 =-\sqrt{AC}\,\langle J_1X,Y\rangle_{g_{A,C}}.
\end{equation}

\begin{lemma}[The horizontal block]\label{lem:D-block-identities}
Set \(T=u_1\). Along each \(u_1\)-orbit, every \(u_1\)-invariant
local section \(X\) of \(\D\) satisfies
\[
\nabla_XT=\nabla_TX=\frac{\sqrt{AC}}2J_1X,
\qquad
R(X,T)T=\frac{AC}{4}X.
\]
\end{lemma}

\begin{proof}
In the invariant frame, \([T,X]=0\).  Pairing Koszul's formula with
\(Y\in\D\) and using \eqref{eq:D-bracket-u1} gives the first identity;
torsion-freeness gives the second.  Since \(J_1X\) is also \(u_1\)-invariant, applying the same formula
to \(J_1X\) and using \(J_1^2=-\Id\) yields the curvature identity.
\end{proof}
\section{\texorpdfstring{The \(S^{4m-1}\times S^3\) construction}{The S4m-1 x S3 construction}}
\label{sec:S4m-times-S3}

We first construct the higher-dimensional family in dimensions
\[
        N=4m+2,\qquad m\ge1.
\]
The \(4m\)-dimensional family will be obtained in the next section as a
totally geodesic fixed point submanifold of these examples.

The underlying closed manifold is
\[
        M_m=S^{4m-1}\times S^3 .
\]
All vector fields and one-forms from the factors will be pulled back to the
product without changing notation.

On \(S^{4m-1}\) we use the Hopf--Berger data from
Section~\ref{sec:hopf-berger-data}.  Thus
\[
        A>C>0,
\]
and the \(g_{A,C}\)-orthonormal Hopf fields
\[
        u_1,u_2,u_3
\]
satisfy
\[
        [u_1,u_2]=A u_3,\qquad
        [u_2,u_3]=A u_1,\qquad
        [u_3,u_1]=C u_2 .
\]
The leaf direction will be
\[
        T=u_1 .
\]

Identify \(S^3\) with the group \(Sp(1)\) of unit quaternions.  On the second
factor choose the bi-invariant metric for which the left-invariant frame
\[
        v_1,v_2,v_3
\]
whose values at the identity are positive multiples of \(i,j,k\), respectively,
is orthonormal and satisfies
\[
        [v_1,v_2]=D v_3,\qquad
        [v_2,v_3]=D v_1,\qquad
        [v_3,v_1]=D v_2,
\]
where \(D>0\).
For the round unit \(S^3\) one has \(D=2\), but the precise value will not
matter.

Let
\[
        \beta_1=g_{A,C}(u_1,\cdot),
        \qquad
        \beta_3=g_{A,C}(u_3,\cdot),
\]
and let
\[
        \sigma_1=g_{S^3}(v_1,\cdot).
\]
These are the dual one-forms for the unperturbed product metric.

Choose real constants
\[
        p,q\in\mathbb R
\]
such that
\[
        p\ne0,\qquad q\ne0,\qquad p^2+q^2<1.
\]
Define a metric \(g\) on
\[
        S^{4m-1}\times S^3
\]
by
\begin{equation}
\label{eq:S4mS3-metric}
\begin{aligned}
        g
        &=
        g_{A,C}
        +
        g_{S^3}
        +
        p(\beta_1\otimes\sigma_1+\sigma_1\otimes\beta_1)       \\
        &\hspace{2cm}
        +
        q(\beta_3\otimes\sigma_1+\sigma_1\otimes\beta_3).
\end{aligned}
\end{equation}
Thus the only cross terms are
\[
        \langle v_1,u_1\rangle=p,
        \qquad
        \langle v_1,u_3\rangle=q .
\]
All other cross terms between the two factors vanish.
We shall denote this metric by \(g_{p,q}\).

The metric is positive definite.  Indeed, the only nontrivial Gram block is
on
\[
        \operatorname{span}\{u_1,u_3,v_1\},
\]
where the matrix is
\[
        \begin{pmatrix}
        1&0&p\\
        0&1&q\\
        p&q&1
        \end{pmatrix}.
\]
Its determinant is
\[
        1-p^2-q^2>0,
\]
and its leading principal minors are positive.

Set
\[
        h=\sqrt{1-p^2-q^2}.
\]
Then
\[
        \nu=\frac{v_1-pu_1-qu_3}{h}
\]
is a unit vector orthogonal to
\[
        u_1,u_2,u_3,v_2,v_3,\D.
\]
Since
\[
        v_1=h\nu+pT+qu_3,
\]
an adapted orthonormal frame is
\[
        T=u_1,\qquad
        \nu,\ u_2,\ u_3,\ v_2,\ v_3,\ \D .
\]

Because vector fields from different factors commute, the relevant brackets
in this frame are the following.  From the Hopf--Berger block,
\begin{equation}
\label{eq:S4mS3-hopf-brackets}
        [T,u_2]=A u_3,\qquad
        [u_2,u_3]=A T,\qquad
        [u_3,T]=C u_2 .
\end{equation}
The brackets involving \(\nu\) and the first factor are
\begin{equation}
\label{eq:S4mS3-nu-hopf-brackets}
        [\nu,T]=-\frac{qC}{h}u_2,\qquad
        [\nu,u_2]=\frac{qA}{h}T-\frac{pA}{h}u_3,\qquad
        [\nu,u_3]=\frac{pC}{h}u_2 .
\end{equation}
The brackets involving the \(S^3\)-directions are
\begin{equation}
\label{eq:S4mS3-S3-brackets}
        [\nu,v_2]=\frac{D}{h}v_3,
        \qquad
        [\nu,v_3]=-\frac{D}{h}v_2,\qquad
        [v_2,v_3]=D(h\nu+pT+qu_3).
\end{equation}

\subsection{The circle fibration}

The foliation \(\F\) is generated by
\[
        T=u_1 .
\]
Since \(u_1\) is a constant multiple of the quaternionic Hopf field
\(\xi_1\), its flow is periodic.  Hence the leaves are closed circles.
The quotient is the smooth fibration
\[
        S^1
        \longrightarrow
        S^{4m-1}\times S^3
        \longrightarrow
        \CP^{2m-1}\times S^3 .
\]
By Section~\ref{sec:hopf-berger-data}, every leaf has length
\begin{equation}
\label{eq:S4mS3-leaf-length}
        \ell_{\mathrm{leaf}}
        =
        \frac{4\pi}{\sqrt{AC}} .
\end{equation}

We check geodesicity.  Since \(T\) is unit, it is enough to prove
\[
        \nabla_TT=0.
\]
For \(Y\perp T\), Koszul's formula gives
\[
        2\langle\nabla_TT,Y\rangle
        =
        -2\langle [T,Y],T\rangle .
\]
Using the brackets above,
\[
        [T,\nu]=\frac{qC}{h}u_2,
        \qquad
        [T,u_2]=A u_3,
        \qquad
        [T,u_3]=-C u_2,
\]
and
\[
        [T,v_2]=[T,v_3]=0.
\]
All these vectors are orthogonal to \(T\).  For \(X\in\D\), we use the
adapted Hopf-invariant frame of Section~\ref{sec:hopf-berger-data}, for
which
\[
        [T,X]=0.
\]
Therefore
\[
        \langle [T,Y],T\rangle=0
        \qquad
        (Y\perp T),
\]
and hence
\[
        \nabla_TT=0.
\]
Thus \(\F\) is a one-dimensional totally geodesic foliation.

\subsection{The mixed curvature operator}

The normal bundle of \(\F\) splits orthogonally as
\[
        T^\perp=E\oplus\D,
\]
where
\[
        E=\operatorname{span}\{\nu,u_2,u_3,v_2,v_3\}.
\]
Let
\[
        R_T:T^\perp\to T^\perp,
        \qquad
        R_T(U)=R(U,T)T
\]
be the mixed curvature operator.

We first record the finite-dimensional connection table.

\begin{lemma}
\label{lem:S4m-S3-connection}
In the finite-dimensional block
\[
        E=\operatorname{span}\{\nu,u_2,u_3,v_2,v_3\},
\]
one has
\[
        \nabla_\nu T
        =
        -\frac{q(A+C)}{2h}u_2,
\]
\[
        \nabla_{u_2}T
        =
        \frac{q(A-C)}{2h}\nu
        -
        \frac{2A-C}{2}u_3,
\]
\[
        \nabla_{u_3}T=\frac C2u_2,
\]
and
\[
        \nabla_{v_2}T=-\frac{Dp}{2}v_3,
        \qquad
        \nabla_{v_3}T=\frac{Dp}{2}v_2.
\]
Moreover,
\[
        \nabla_T\nu
        =
        -\frac{q(A-C)}{2h}u_2,
\]
\[
        \nabla_Tu_2
        =
        \frac{q(A-C)}{2h}\nu
        +
        \frac C2u_3,
\]
\[
        \nabla_Tu_3=-\frac C2u_2,
\]
and
\[
        \nabla_Tv_2=-\frac{Dp}{2}v_3,
        \qquad
        \nabla_Tv_3=\frac{Dp}{2}v_2.
\]
These covariant derivatives have no \(\D\)-component.
\end{lemma}

\begin{proof}
All inner products of the adapted frame
\[
        T,\nu,u_2,u_3,v_2,v_3
\]
are constant.  Therefore Koszul's formula reduces to
\[
        2\langle \nabla_EF,G\rangle
        =
        \langle [E,F],G\rangle
        -
        \langle [F,G],E\rangle
        +
        \langle [G,E],F\rangle
\]
for \(E,F,G\) in this finite frame.  Substituting
\eqref{eq:S4mS3-hopf-brackets},
\eqref{eq:S4mS3-nu-hopf-brackets}, and
\eqref{eq:S4mS3-S3-brackets}, together with
\[
        [v_2,v_3]=D(h\nu+pT+qu_3),
\]
gives the displayed coefficients.

The absence of \(\D\)-components follows from two facts.  First, vector
fields on the \(S^3\)-factor commute with vector fields on the
\(S^{4m-1}\)-factor.  Second, brackets of Hopf fields with \(\D\)-fields have
no component in the Hopf three-plane, as recorded in
Section~\ref{sec:hopf-berger-data}.
\end{proof}

\begin{lemma}
\label{lem:S4m-S3-curvature-block}
Relative to
\(E=\spanop\{\nu,u_2,u_3\}\oplus\spanop\{v_2,v_3\}\), one has
\[
 R_T|_E=R_3\oplus\frac{D^2p^2}{4}\Id_{\spanop\{v_2,v_3\}},
\]
where, in the basis \(\nu,u_2,u_3\),
\[
R_3=
\begin{pmatrix}
\dfrac{q^2(A-C)(A+3C)}{4h^2}
&0&-\dfrac{3qC(A-C)}{4h}\\[7pt]
0&\dfrac{C^2h^2+q^2(A-C)^2}{4h^2}&0\\[7pt]
-\dfrac{3qC(A-C)}{4h}&0&\dfrac{C(4A-3C)}4
\end{pmatrix}.
\]
Moreover,
\[
        R_T|_{\D}=\frac{AC}{4}\Id_{\D},
\]
and there are no mixed entries between \(E\) and \(\D\).
\end{lemma}

\begin{proof}
Since
\[
        \nabla_TT=0,
\]
we use
\[
        R(U,T)T
        =
        -
        \nabla_T\nabla_UT
        -
        \nabla_{[U,T]}T .
\]
The connection coefficients in Lemma~\ref{lem:S4m-S3-connection} give
\[
        R(\nu,T)T
        =
        \frac{q^2(A-C)(A+3C)}{4h^2}\nu
        -
        \frac{3qC(A-C)}{4h}u_3,
\]
\[
        R(u_2,T)T
        =
        \frac{C^2h^2+q^2(A-C)^2}{4h^2}u_2,
\]
\[
        R(u_3,T)T
        =
        -
        \frac{3qC(A-C)}{4h}\nu
        +
        \frac{C(4A-3C)}4u_3,
\]
and
\[
        R(v_2,T)T
        =
        \frac{D^2p^2}{4}v_2,
        \qquad
        R(v_3,T)T
        =
        \frac{D^2p^2}{4}v_3.
\]
This gives the displayed finite-dimensional block.

It remains to check the \(\mathcal D\)-block for the metric
\eqref{eq:S4mS3-metric} on the product.
Lemma~\ref{lem:D-block-identities} gives the
intrinsic Hopf--Berger identities on the sphere factor.  We now verify that
the \(S^3\)-coupling does not change them.

The metric differs from the product metric only in the finite-dimensional
block
\[
        \operatorname{span}\{T,u_2,u_3,v_1,v_2,v_3\}.
\]
The distribution \(\D\) remains orthogonal to this block, and the restriction
of \(g\) to \(\D\) is still \(g_{A,C}|_{\D}\).  The \(S^3\)-fields commute
with fields on \(S^{4m-1}\), and brackets of Hopf fields with \(\D\)-fields
have no Hopf component.

Let \(X,Y\in\D\), and take the adapted Hopf-invariant frame for which
\[
        [T,X]=[T,Y]=0.
\]
If \(W\in E\), then Koszul's formula gives
\[
        \langle \nabla_XT,W\rangle=0,
        \qquad
        \langle \nabla_TX,W\rangle=0.
\]
Testing against \(Y\in\D\) gives exactly the intrinsic Hopf--Berger
calculation from Lemma~\ref{lem:D-block-identities}.  Hence, for \(g\),
\[
\nabla_XT=\frac{\sqrt{AC}}2J_1X,
\qquad
\nabla_TX=\frac{\sqrt{AC}}2J_1X.
\]
Therefore
\[
        R(X,T)T
        =
        \frac{AC}{4}X.
\]
The finite-dimensional calculation above gives
\[
        R(E,T)T\subset E,
\]
while the last formula gives
\[
        R(\D,T)T\subset\D.
\]
Thus there are no mixed matrix entries between \(E\) and \(\D\).
\end{proof}

We now check positivity.  Since
\[
        A>C>0,\qquad q\ne0,\qquad h>0,
\]
the middle diagonal entry
\[
        \frac{C^2h^2+q^2(A-C)^2}{4h^2}
\]
is positive.  The determinant of the \((\nu,u_3)\)-block is
\[
\begin{aligned}
        &
        \frac{q^2(A-C)(A+3C)}{4h^2}
        \cdot
        \frac{C(4A-3C)}4
        -
        \left(
        \frac{3qC(A-C)}{4h}
        \right)^2
        \\
        &=
        \frac{q^2A^2C(A-C)}{4h^2}
        >
        0.
\end{aligned}
\]
The \(v_2\)- and \(v_3\)-entries are positive because
\[
        \frac{D^2p^2}{4}>0.
\]
Finally,
\[
        \frac{AC}{4}>0
\]
on \(\D\).  Therefore
\[
        R_T>0
\]
on the whole normal bundle \(T^\perp\).  Since the foliation is
one-dimensional, this is equivalent to
\[
        \Kmix>0.
\]

We have proved the following.

\begin{proposition}
\label{prop:S4m-S3}
For every \(m\ge1\), every \(A>C>0\), every \(D>0\), and every
\[
        p,q\in\mathbb R
\]
satisfying
\[
        p\ne0,\qquad q\ne0,\qquad p^2+q^2<1,
\]
the metric \eqref{eq:S4mS3-metric} on
\[
        S^{4m-1}\times S^3
\]
carries a one-dimensional totally geodesic foliation by closed circles with
positive mixed sectional curvature.  The leaves have length
\[
        \ell_{\mathrm{leaf}}=\frac{4\pi}{\sqrt{AC}} .
\]
\end{proposition}

\subsection{A parallel anisotropic subfamily}
\label{subsec:parallel-anisotropic}

The preceding formulas allow us to distinguish exactly when the mixed
curvature operator rotates and when it is parallel.  Set
\[
        r=\frac qh,
        \qquad
        d=A-C>0,
\]
and consider the block
\[
        E_0=\operatorname{span}\{\nu,u_2,u_3\}.
\]
In the ordered basis \(\nu,u_2,u_3\), Lemmas
\ref{lem:S4m-S3-connection} and \ref{lem:S4m-S3-curvature-block} give the
normal connection matrix
\[
\Omega_0
=
\frac12
\begin{pmatrix}
0&rd&0\\
-rd&0&-C\\
0&C&0
\end{pmatrix}
\]
and the mixed curvature matrix
\[
R_0
=
\frac14
\begin{pmatrix}
r^2d(A+3C)&0&-3rCd\\
0&C^2+r^2d^2&0\\
-3rCd&0&C(4A-3C)
\end{pmatrix}.
\]
A direct multiplication yields
\begin{equation}
\label{eq:curvature-commutator}
[\Omega_0,R_0]
=
\frac{Cd(C-r^2d)}2
\begin{pmatrix}
0&r&0\\
r&0&-1\\
0&-1&0
\end{pmatrix}.
\end{equation}
The \(v_2,v_3\)-curvature block is scalar, as is the \(\D\)-block, and the
normal connection preserves all these blocks.  Consequently, within the
family of Proposition~\ref{prop:S4m-S3},
\begin{equation}
\label{eq:parallel-condition}
        \nabla_T^\perp R_T=0
        \quad\Longleftrightarrow\quad
        r^2(A-C)=C.
\end{equation}

\begin{proposition}[Parallel anisotropic subfamily]
\label{prop:parallel-anisotropic-S3}
Fix \(m\ge1\) and \(A>C>0\).  Put
\[
        r=\sqrt{\frac{C}{A-C}}.
\]
For any \(p\in(0,1)\), set
\[
        h=\sqrt{\frac{1-p^2}{1+r^2}},
        \qquad
        q=rh,
        \qquad
        D=\frac{\sqrt{AC}}{p}.
\]
Then the metric \eqref{eq:S4mS3-metric} satisfies the hypotheses of
Proposition~\ref{prop:S4m-S3}, and its mixed curvature operator is positive,
parallel along every leaf, and nonscalar.  At every point,
\[
        \operatorname{spec}(R_T)
        =
        \left\{\frac{AC}{4},\,AC\right\},
\]
where \(AC\) is simple and \(AC/4\) has multiplicity \(4m\).
\end{proposition}

\begin{proof}
The definitions give
\[
        h^2=1-p^2-q^2>0,
        \qquad
        \frac qh=r,
\]
so the metric is well defined and \eqref{eq:parallel-condition} holds.  Hence
\(R_T\) is parallel.  Under the relation \(r^2(A-C)=C\), the characteristic
polynomial of the three-dimensional block factors as
\[
        \det(\lambda\Id-R_0)
        =
        \left(\lambda-AC\right)
        \left(\lambda-\frac{AC}{4}\right)^2.
\]
The \(\D\)-block is \(AC/4\), and the choice of \(D\) makes both
\(v_2,v_3\)-entries equal to \(AC/4\).  This proves the asserted spectrum
and, in particular, positivity and anisotropy.
\end{proof}

The breakdown of the Ferus argument is visible on the real eigenline of the
conullity tensor.  Choose the parallel frame to agree with the moving frame at
one point.  Under \eqref{eq:parallel-condition}, the \((\nu,u_2,u_3)\)-block
of \(B(0)\) is
\[
 B_0=\frac12
 \begin{pmatrix}
 0&r(A-C)&0\\
 -r(A+C)&0&C\\
 0&-(2A-C)&0
 \end{pmatrix}.
\]
Its characteristic polynomial is
\[
 \det(\lambda\Id-B_0)
 =\lambda\left(\lambda^2+\frac{3AC}{4}\right).
\]
Its unique real eigenline is generated by
\[
        Z=r\nu+(1+2r^2)u_3,
        \qquad B_0Z=0.
\]
On the other hand,
\[
        R_0Z=\frac{AC}{2}\bigl(-r\nu+(r^2+2)u_3\bigr)
        \notin\R Z.
\]
The Riccati equation therefore gives
\(B'(0)Z=-R_0Z\notin\R Z\).  Thus the real eigenline forced by odd normal
rank is not preserved by the evolution.  This is the point at which
scalarity, rather than parallelism, enters the Ferus argument.

\begin{remark}\label{rem:two-couplings}
The two couplings have different roles.  The \(q\)-term makes the
\((\nu,u_3)\)-block positive, whereas the \(p\)-term, through
\([v_2,v_3]=Dv_1\), creates the positive entries
\(R_T(v_2)=R_T(v_3)=D^2p^2/4\).  Thus neither coupling can be omitted.
\end{remark}

Finally, the metric is not bundle-like.  Indeed,
Lemma~\ref{lem:S4m-S3-connection} gives
\[
 \langle B\nu,u_2\rangle+\langle\nu,Bu_2\rangle
 =-\frac{qC}{h}\ne0.
\]
By \eqref{eq:Lie-B}, the transverse metric is not invariant under the circle
flow.  This completes the \(4m+2\)-dimensional construction.

\section{The fixed point examples}
\label{sec:S4m-times-S1}

View \(S^3\) as the unit quaternions and choose \(v_1\) in the \(i\)-direction.
Conjugation by \(i\),
\[
 \tau(a+bi+cj+dk)=a+bi-cj-dk,
\]
fixes \(v_1\), reverses \(v_2,v_3\), and preserves the bi-invariant metric
and \(\sigma_1\).  Hence \(\Id\times\tau\) is an isometry of
\((S^{4m-1}\times S^3,g_{p,q})\).  Its fixed component
\[
        P_m=S^{4m-1}\times S_i^1
\]
is totally geodesic.

The field \(T=u_1\) is tangent to \(P_m\) and generates
\[
        S^1\longrightarrow P_m\longrightarrow \CP^{2m-1}\times S_i^1.
\]
If \(\sigma=\sigma_1|_{S_i^1}\), the induced metric is
\begin{equation}\label{eq:fixed-metric}
\begin{aligned}
 g_P={}&g_{A,C}+\sigma^2
 +p(\beta_1\otimes\sigma+\sigma\otimes\beta_1)\\
 &+q(\beta_3\otimes\sigma+\sigma\otimes\beta_3).
\end{aligned}
\end{equation}
Because \(P_m\) is totally geodesic, both \(\nabla_TT=0\) and the mixed
curvatures of tangent planes agree with those of the ambient construction.

\begin{proposition}\label{prop:S4m-S1}
For every \(m\ge1\), \(A>C>0\), \(D>0\), and
\(p,q\ne0\) with \(p^2+q^2<1\), the metric \eqref{eq:fixed-metric} on
\(S^{4m-1}\times S^1\) carries a one-dimensional totally geodesic foliation
by closed circles with \(\Kmix>0\).  The leaves have length
\(4\pi/\sqrt{AC}\), and the metric is not bundle-like.
\end{proposition}

\begin{proof}
Only the last assertion remains.  Since \(P_m\) is totally geodesic, the
ambient connection table restricts to it, and
\[
 \langle B_P\nu,u_2\rangle+\langle\nu,B_Pu_2\rangle
 =-\frac{qC}{h}\ne0.
\]
Equation \eqref{eq:Lie-B} applies.
\end{proof}

\begin{corollary}\label{cor:parallel-anisotropic-S1}
For the parameters in Proposition~\ref{prop:parallel-anisotropic-S3}, the
mixed-curvature operator of \eqref{eq:fixed-metric} is positive and parallel,
with spectrum \(\{AC/4,AC\}\); \(AC\) is simple and \(AC/4\) has multiplicity
\(4m-2\).
\end{corollary}

\begin{proof}
The normal bundle is
\(\spanop\{\nu,u_2,u_3\}\oplus\D\), preserved by normal parallel transport,
and the curvature operator is the restriction of the ambient one.
\end{proof}
Sections~\ref{sec:S4m-times-S3} and~\ref{sec:S4m-times-S1}
prove Theorem~\ref{thm:main}, while
Proposition~\ref{prop:parallel-anisotropic-S3} and
Corollary~\ref{cor:parallel-anisotropic-S1}
prove Theorem~\ref{thm:parallel-anisotropic}.

\section{Pinching and the size of the symmetric conullity}
\label{sec:pinching}

For a one-dimensional totally geodesic foliation generated by a unit field
\(T\), let
\[
        R_T(U)=R(U,T)T
\]
be the mixed curvature operator on \(T^\perp\).  Since the manifolds are
closed and \(R_T\) is positive in our examples, we may set
\[
        \lambda_{\min}
        =
        \min_{x\in M}\lambda_{\min}(R_T(x)),
        \qquad
        \lambda_{\max}
        =
        \max_{x\in M}\lambda_{\max}(R_T(x)).
\]
Then
\[
        \delta_{\mathrm{mix}}
        =
        \frac{\lambda_{\min}}{\lambda_{\max}}
\]
is the mixed curvature pinching constant.

\subsection{An almost scalar but rotating family}

We use the \(S^{4m-1}\times S^3\) construction of
Section~\ref{sec:S4m-times-S3}.  Fix once and for all a number
\[
        p_0\in(0,1).
\]
For \(\varepsilon>0\), set
\[
        C_\varepsilon=1,
        \qquad
        A_\varepsilon=1+\varepsilon.
\]
Define
\[
        r_\varepsilon^2
        =
        \frac{A_\varepsilon C_\varepsilon}
        {(A_\varepsilon-C_\varepsilon)(A_\varepsilon+3C_\varepsilon)}
        =
        \frac{1+\varepsilon}{\varepsilon(4+\varepsilon)}.
\]
We choose \(h_\varepsilon\) and \(q_\varepsilon\) by
\[
        h_\varepsilon
        =
        \sqrt{\frac{1-p_0^2}{1+r_\varepsilon^2}},
        \qquad
        q_\varepsilon=r_\varepsilon h_\varepsilon .
\]
Then
\[
        h_\varepsilon^2
        =
        1-p_0^2-q_\varepsilon^2,
\]
and hence
\[
        p_0^2+q_\varepsilon^2<1.
\]
Finally set
\[
        D_\varepsilon
        =
        \frac{\sqrt{A_\varepsilon C_\varepsilon}}{p_0}
        =
        \frac{\sqrt{1+\varepsilon}}{p_0}.
\]
With these choices,
\[
        \frac{q_\varepsilon}{h_\varepsilon}=r_\varepsilon,
        \qquad
        \frac{D_\varepsilon^2p_0^2}{4}
        =
        \frac{A_\varepsilon C_\varepsilon}{4}
        =
        \frac{1+\varepsilon}{4}.
\]

The parallelism condition \eqref{eq:parallel-condition} is not satisfied:
\[
        r_\varepsilon^2(A_\varepsilon-C_\varepsilon)
        =
        \frac{1+\varepsilon}{4+\varepsilon}
        \ne 1=C_\varepsilon.
\]
Thus, for every \(\varepsilon>0\), this is a genuinely rotating rather than
a parallel subfamily.

Let \(g_\varepsilon\) be the metric \eqref{eq:S4mS3-metric} with parameters
\[
        A=A_\varepsilon,\quad C=C_\varepsilon,\quad
        p=p_0,\quad q=q_\varepsilon,\quad D=D_\varepsilon.
\]
By Proposition~\ref{prop:S4m-S3}, it has positive mixed sectional curvature
on
\[
        S^{4m-1}\times S^3.
\]

We now examine the matrix of \(R_T\).  In the finite block
\[
        \operatorname{span}\{\nu,u_2,u_3\},
\]
Lemma~\ref{lem:S4m-S3-curvature-block} gives
\[
R_{\varepsilon}^{(3)}
=
\begin{pmatrix}
\dfrac{1+\varepsilon}{4}
&
0
&
-\dfrac{3\varepsilon r_\varepsilon}{4}
\\[8pt]
0
&
\dfrac{1+\varepsilon^2r_\varepsilon^2}{4}
&
0
\\[8pt]
-\dfrac{3\varepsilon r_\varepsilon}{4}
&
0
&
\dfrac{1+4\varepsilon}{4}
\end{pmatrix}.
\]
Here
\[
        \varepsilon r_\varepsilon
        =
        \sqrt{\frac{\varepsilon(1+\varepsilon)}{4+\varepsilon}},
\]
and
\[
        \varepsilon^2r_\varepsilon^2
        =
        \frac{\varepsilon(1+\varepsilon)}{4+\varepsilon}.
\]
Therefore
\[
        R_{\varepsilon}^{(3)}
        \longrightarrow
        \frac14 \operatorname{Id}
        \qquad
        \text{as }\varepsilon\to0.
\]
The \(v_2\)- and \(v_3\)-entries are
\[
        \frac{D_\varepsilon^2p_0^2}{4}
        =
        \frac{1+\varepsilon}{4},
\]
and the \(\D\)-block is
\[
        \frac{A_\varepsilon C_\varepsilon}{4}\operatorname{Id}_{\D}
        =
        \frac{1+\varepsilon}{4}\operatorname{Id}_{\D}.
\]
Consequently, on the whole normal bundle,
\[
        R_T^{g_\varepsilon}
        \longrightarrow
        \frac14\operatorname{Id}
\]
in operator norm, uniformly on the manifold.  Hence
\[
        \lambda_{\min}(g_\varepsilon)\to\frac14,
        \qquad
        \lambda_{\max}(g_\varepsilon)\to\frac14,
\]
and therefore
\[
        \delta_{\mathrm{mix}}(g_\varepsilon)\to1.
\]

The same conclusion holds for the \(4m\)-dimensional fixed point examples of
Section~\ref{sec:S4m-times-S1}.  Indeed, their mixed curvature operator is
the restriction of the above operator to
\[
        \operatorname{span}\{\nu,u_2,u_3\}\oplus\D,
\]
and that restriction also converges to
\[
        \frac14\operatorname{Id}.
\]

Thus we obtain:

\begin{proposition}
\label{prop:pinching}
For every even integer \(N\ge4\) and every \(0<\delta<1\), the examples in
Propositions~\ref{prop:S4m-S3} and~\ref{prop:S4m-S1} can be chosen so that
\[
        \delta_{\mathrm{mix}}>\delta .
\]
\end{proposition}

The lengths of the leaves are also controlled.  In the above family,
\[
        \ell_{\mathrm{leaf}}(g_\varepsilon)
        =
        \frac{4\pi}{\sqrt{A_\varepsilon C_\varepsilon}}
        =
        \frac{4\pi}{\sqrt{1+\varepsilon}}
        \longrightarrow
        4\pi.
\]
If we normalize the metrics by
\[
        \widehat g_\varepsilon
        =
        \lambda_{\max}(g_\varepsilon)\,g_\varepsilon,
\]
then
\[
        \max \Kmix(\widehat g_\varepsilon)=1,
\]
and the leaf lengths become
\[
        \ell_{\mathrm{leaf}}(\widehat g_\varepsilon)
        =
        \sqrt{\lambda_{\max}(g_\varepsilon)}
        \,
        \ell_{\mathrm{leaf}}(g_\varepsilon)
        \longrightarrow
        2\pi.
\]
In particular, the leaf lengths remain uniformly bounded in the almost
\(1\)-pinched regime.

Here \(h_\varepsilon\to0\), so the full metrics need not converge to a
nondegenerate limit; the assertion concerns the mixed-curvature operators
and the normalized leaf lengths.

This proves Theorem~\ref{thm:pinched}.

\subsection{The symmetric part of the conullity tensor}

We now record the size of the non-bundle-like part of the examples.  Let
\[
        B(U)=\nabla_U T
\]
be the conullity tensor and set
\[
        B^{\mathrm{sym}}
        =
        \frac12(B+B^{\mathsf T}).
\]
We use the operator norm and denote it by
\[
        \|\cdot\|_{\operatorname{op}}.
\]

In both the \(S^{4m-1}\times S^3\) construction and its fixed
\(S^{4m-1}\times S^1\) submanifold, the symmetric part of \(B\) is supported
on the block
\[
        \operatorname{span}\{\nu,u_2,u_3\}.
\]
Writing
\[
        r=\frac qh,
\]
and using the ordered basis
\[
        \nu,\ u_2,\ u_3,
\]
we have
\[
B^{\mathrm{sym}}
=
\frac12
\begin{pmatrix}
0&-rC&0\\[7pt]
-rC&0&-(A-C)\\[7pt]
0&-(A-C)&0
\end{pmatrix}.
\]
Hence
\begin{equation}
\label{eq:Bsym-operator-norm}
        \|B^{\mathrm{sym}}\|_{\operatorname{op}}^2
        =
        \frac{r^2C^2+(A-C)^2}{4}.
\end{equation}
The Hilbert--Schmidt norm satisfies
\[
        \|B^{\mathrm{sym}}\|_{\mathrm{HS}}^2
        =
        \frac{r^2C^2+(A-C)^2}{2}.
\]

For the almost pinched family above,
\[
        C=1,
        \qquad
        A=1+\varepsilon,
        \qquad
        r^2=r_\varepsilon^2
        =
        \frac{1+\varepsilon}{\varepsilon(4+\varepsilon)}.
\]
Therefore
\[
        \|B^{\mathrm{sym}}\|_{\operatorname{op}}^2
        =
        \frac14
        \left(
        \frac{1+\varepsilon}{\varepsilon(4+\varepsilon)}
        +
        \varepsilon^2
        \right)
        \sim
        \frac{1}{16\varepsilon}.
\]
On the other hand,
\[
        \lambda_{\min}(g_\varepsilon)\to\frac14.
\]
Thus
\[
        \frac{
        \|B^{\mathrm{sym}}\|_{\operatorname{op}}^2
        }{
        \lambda_{\min}(g_\varepsilon)
        }
        \longrightarrow
        \infty.
\]
So the almost \(1\)-pinched examples are far from bundle-like in a
quantitative sense.  The curvature operator becomes almost scalar positive,
but the symmetric part of the conullity tensor becomes large.

There is also a sharp lower bound for the ratio
\[
        \frac{
        \|B^{\mathrm{sym}}\|_{\operatorname{op}}^2
        }{
        \min\Kmix
        }
\]
within the present family.

\begin{proposition}
\label{prop:Bsym-threshold}
For the examples constructed in Sections~\ref{sec:S4m-times-S3} and
\ref{sec:S4m-times-S1},
\[
        \frac{
        \|B^{\mathrm{sym}}\|_{\operatorname{op}}^2
        }{
        \min\Kmix
        }
        \ge 1.
\]
Moreover, the constant \(1\) is sharp in this family.  If the
Hilbert--Schmidt norm is used instead of the operator norm, the corresponding
sharp constant is \(2\).
\end{proposition}

\begin{proof}
The \(v_2,v_3\)-block of the \(S^{4m-1}\times S^3\) examples contributes no
symmetric part to \(B\).  It can only decrease \(\min\Kmix\), and hence can
only increase the ratio above.  Thus it is enough to consider the
\((\nu,u_2,u_3)\)-block.

Set
\[
        \varepsilon=\frac{A-C}{C}>0,
        \qquad
        r=\frac qh.
\]
Then \eqref{eq:Bsym-operator-norm} becomes
\[
        \|B^{\mathrm{sym}}\|_{\operatorname{op}}^2
        =
        \frac{C^2}{4}(r^2+\varepsilon^2).
\]
In the \((\nu,u_3)\)-block of \(R_T\), take the vector
\[
        W=\nu+r u_3.
\]  
A direct substitution into the
displayed curvature matrix gives
\[
        \frac{\langle R_TW,W\rangle}{|W|^2}
        =
        \frac{C^2}{4}(r^2+\varepsilon^2)
        -
        \frac{C^2}{4}
        \frac{(r^2-\varepsilon)^2}{1+r^2}.
\]
Hence
\[
        \min\Kmix
        \le
        \frac{\langle R_TW,W\rangle}{|W|^2}
        \le
        \frac{C^2}{4}(r^2+\varepsilon^2)
        =
        \|B^{\mathrm{sym}}\|_{\operatorname{op}}^2.
\]
This proves the inequality.

Sharpness follows by taking
\[
        A=2C,
        \qquad
        r=1.
\]
Then the \((\nu,u_3)\)-block has smallest eigenvalue
\[
        \frac{C^2}{2},
\]
the \(u_2\)-entry is also
\[
        \frac{C^2}{2},
\]
and the \(\D\)-block is
\[
        \frac{AC}{4}=\frac{C^2}{2}.
\]
In the \(S^{4m-1}\times S^3\) case, choose \(D\) so that
\[
        \frac{D^2p^2}{4}\ge\frac{C^2}{2}.
\]
Then
\[
        \min\Kmix
        =
        \frac{C^2}{2}
        =
        \|B^{\mathrm{sym}}\|_{\operatorname{op}}^2.
\]
The statement for the Hilbert--Schmidt norm follows from
\[
        \|B^{\mathrm{sym}}\|_{\mathrm{HS}}^2
        =
        2\|B^{\mathrm{sym}}\|_{\operatorname{op}}^2
\]
on this block.  Notice that the equality choice \(A=2C\), \(r=1\) also
satisfies \(r^2(A-C)=C\); the sharp case therefore belongs to the parallel
anisotropic subfamily of Proposition~\ref{prop:parallel-anisotropic-S3}.
\end{proof}

\section{A quantitative refinement}
\label{sec:consequences}

The preceding constructions show that neither leafwise parallelism of
\(R_T\) nor arbitrarily strong mixed-curvature pinching short of equality
restores the Ferus--Adams bound in the non-bundle-like setting.
This suggests measuring the failure
of transverse metric invariance through the symmetric conullity.

For a totally geodesic foliation \(\F^p\) and a unit
\(X\in T\F\), set
\[
 B_X(U)=(\nabla_UX)^{\Hh},\qquad
 B_X^{\mathrm{sym}}=\frac12(B_X+B_X^{\mathsf T}),
\]
and define
\[
 \mu=\inf\Kmix,
 \qquad
 \sigma=\sup_{x\in M,\,X\in T_x\F,\,|X|=1}
 \|B_X^{\mathrm{sym}}\|_{\operatorname{op}}^2.
\]
The ratio \(\sigma/\mu\) is invariant under constant rescaling.  For a
one-dimensional geodesic foliation, \(\sigma=0\) is exactly the bundle-like
condition by \eqref{eq:Lie-B}; Proposition~\ref{prop:bundle-like-Ferus} then
recovers the Ferus--Adams bound.

\begin{problem}[Quantitative Toponogov problem]
\label{prob:quantitative-toponogov}
Does there exist a universal constant \(c_*>0\) such that every closed
Riemannian manifold with a totally geodesic foliation satisfying
\[
        \Kmix>0,
        \qquad
        \sigma<c_*\mu
\]
obeys \(p\le\rh(n)-1\)?  If so, what is the optimal \(c_*\)?
\end{problem}

Proposition~\ref{prop:Bsym-threshold} implies \(c_*\le1\) for the operator
norm: the present family contains counterexamples with \(\sigma/\mu=1\).
For the Hilbert--Schmidt norm the corresponding upper bound is \(2\).  The
comparison must use \(\mu=\inf\Kmix\), rather than \(\sup\Kmix\): the
parameter \(D\) changes
\[
        R_T(v_2)=R_T(v_3)=\frac{D^2p^2}{4}
\]
without changing \(B^{\mathrm{sym}}\), so \(\sup\Kmix\) can be made
arbitrarily large at fixed \(\sigma\).  
The examples therefore suggest \(\sigma=\mu\) as the natural
borderline, while the bundle-like case lies at \(\sigma=0\).
Counterexamples with higher-dimensional leaves remain a
separate problem.

\section*{Acknowledgements}
The author is grateful to V. A. Sharafutdinov for bringing Toponogov's
question on mixed curvature and Rovenskii's related work to his attention.

\end{document}